\documentclass[12pt]{article}

\usepackage{t1enc,amssymb,amsbsy,amsthm,amsmath,amsfonts}

\newcommand{\ol}{\overline}

\newtheorem{theorem}{Theorem}[section]
\newtheorem{lemma}[theorem]{Lemma}
\newtheorem{proposition}[theorem]{Proposition}

\newtheorem{example}[theorem]{Example}
\newtheorem{definition}[theorem]{Definition}
\newtheorem{remark}[theorem]{Remark}

\numberwithin{equation}{section}

\def\ln{\hbox {\rm ln\,}}

\def\P{{\bf P}}
\def\R{{\bf R}}
\def\R{{\bf R}}
\def\Q{{\bf Q}}
\def\E{{\bf E}}
\def\ep{{\epsilon}}
\def\cE{{\cal E}}
\def\cF{{\cal F}}

\def\cC{{\cal C}}

\def\rp{\right)}
\def\lp{\left(}

\def\disp{\displaystyle{}}

\def\al{\alpha}
\def\ep{\varepsilon}

\def\ga{\gamma}

\def\de{\delta}
\def\te{\theta}
\def\be{\beta}
\def\la{\lambda}

\begin{document}

\author{Paavo Salminen\\{\small Åbo Akademi University}
\\{\small Mathematical Department}
\\{\small FIN-20500 Åbo, Finland} \\{\small email: phsalmin@abo.fi}
\and
Pierre Vallois
\\{\small Universit\'e Henri Poincar\'e}
\\{\small D\'epartement de Math\'ematique}
\\{\small F-54506 Vandoeuvre les Nancy, France}
\\{\small email: vallois@iecn.u-nancy.fr}
}

\title{On subexponentiality of the L\'evy measure of the diffusion inverse local time; with applications to penalizations}
\date{}
\maketitle

\begin{abstract} For a recurrent linear diffusion on $\R_+$ 
we study the asymptotics of the distribution of its local time at 0 as
the time parameter tends to infinity. Under the assumption 
that the L\'evy measure of the inverse local time is subexponential 
this distribution behaves asymtotically as a multiple of the L\'evy
measure. Using spectral representations we find the exact value of the
multiple. For this we also need a result on the asymptotic
behavior of the convolution of a
subexponential distribution and an arbitrary distribution on $\R_+.$
The exact knowledge of the asymptotic behavior of the distribution of the
local time allows us to analyze the process
derived via a penalization procedure with the local time. This result
generalizes the penalizations obtained in Roynette, Vallois and
Yor \cite{rvyV} for Bessel processes. 
\\
\\
\\
\\ \\
{\rm Keywords: Brownian motion, Bessel process, Hitting time,
  Tauberian theorem, excursions}  
\\ \\ 
{\rm AMS Classification: 60J60, 60J65, 60J30} 
\end{abstract}

\eject
\section{Introduction}
\label{sec0}
{\bf 1.} Let $X$ be a linear regular recurrent 
diffusion taking values in $\R_+$ with 0 an instantaneously reflecting
boundary and $+\infty$ a natural boundary. Let $\P_x$ and $\E_x$ 
denote, respectively, the probability measure and the
expectation associated with $X$ when started from $x\geq 0.$ We assume that $X$ is 
defined in the canonical space $C$ of continuous functions $\omega:\R_+\mapsto \R_+.$ 
Let 
$$
{\cal C}_t:=\sigma\{\omega(s): s\leq t\}
$$  
denote the smallest $\sigma$-algebra making the co-ordinate mappings up to time $t$ measurable 
and take ${\cal C}$ to be the smallest $\sigma$-algebra including all $\sigma$-algebras ${\cal C}_t,\ t\geq 0.$

We let $m$ and $S$ denote the speed measure and the scale function of
$X,$ respectively. We normalize $S$ by $S(0)=0$ and remark that
$S(+\infty)=+\infty$ since we assume $X$ to be recurrent. 
It is also assumed that $m$ does not have atoms. Recall that 
$X$ has a jointly continuous transition density 
$p(t;x,y)$ with respect to  $m,$ i.e., 
$$
\P_x(X_t\in A)=\int_A p(t;x,y)\, m(dy),
$$
where $A$ is a Borel subset of $\R_+.$ Moreover, $p$ is symmetric in $x$ and $y,$ that is,
$p(t;x,y)=p(t;y,x).$ The Green or the resolvent kernel of $X$ is defined for $\lambda>0$ via
\begin{equation}
\label{a0}
 R_\lambda(x,y):=\int_0^\infty {\rm e}^{-\lambda t}\,p(t;x,y)\,dt ,
\end{equation}

Let $\{L^{(y)}_t\,:\, t\geq 0\}$ denote the local time of $X$ at $y$
normalized via 
\begin{equation}
\label{e000}
L^{(y)}_t=\lim_{\delta\downarrow 0}\frac 1{m((y,y+\delta))} \int_0^t 
{\bf 1}_{[y,y+\delta)}(X_s)\, ds.
\end{equation}
For $y=0$ we write simply $L_t,$ and define for $\ell\geq 0$ 
\begin{equation}
\label{e00}
\tau_\ell:=\inf\{s: L_s>\ell\},
\end{equation}
i.e., $\tau:=\{\tau_\ell:\ell\geq 0\}$ is the right continuous inverse of $\{L_t\}.$ As is well known
$\tau$ 
is an increasing L\'evy process, in other words, a subordinator
and its L\'evy exponent is given by
\begin{eqnarray}
\label{e1}
&&\hskip-1cm
\nonumber
\E_0\left(\exp(-\lambda \tau_\ell)\right)=\exp\left(-\ell/R_\lambda(0,0)\right)
\\
&&\hskip1.8cm=
\exp(-\ell\int_0^\infty \nu(dv)(1-{\rm e}^{-\lambda v})),
\end{eqnarray}
where $\nu$ is the L\'evy measure of $\tau.$ 
The assumption that the speed measure does not have an atom at 0 implies that
$\tau$ does not have a drift. 
\vskip.4cm
\noindent
{\bf 2.} We are interested in the asymptotic behavior of the distribution of
$L_t$ as $t$ tends to infinity. The basic assumption under which this
study is done is the subexponentiality of the L\'evy measure of $\tau$
(see Section \ref{sec3}). The subexponentiality assumption is equivalent with the relation  
(cf. Proposition \ref{prop31}) 
$$
\P(\tau_\ell\geq t) \,\mathop{\sim}_{t\to +\infty}\,
\ell\,\nu((t,+\infty))\quad  \forall\ \ell>0.
$$

Here and throughout the paper the notation 
$$
f(x)
\,\mathop{\sim}_{x\to a}\, 
g(x),
$$
where $f$ and $g$ are real valued functions and $a$ is allowed to take 
also ``values'' $+\infty$ or $-\infty,$ means that 
$$
\lim_{x\to a}\frac{f(x)}{g(x)}=1.
$$

Since $\tau$ is the inverse of $L,$ it also holds (see Proposition \ref{prop31}) 
$$
 \P_0(L_t\leq \ell) \,\mathop{\sim}_{t\to +\infty}\, \ell\,\nu((t,+\infty)).
$$
To extend this for an arbitrary starting state $x>0,$ we first show
that  (see Proposition \ref{prop32}) 
$$
 \P_x(H_0>t) \,\mathop{\sim}_{t\to +\infty}\, S(x)\,\nu((t,+\infty)),
$$
where 
$
H_0:=\inf\{t:\, X_t=y\},
$
and then  (see Proposition \ref{prop33}) 
\begin{equation}
\label{f04}
 \P_x(L_t\leq \ell) \,\mathop{\sim}_{t\to +\infty}\, (S(x)+\ell)\,\nu((t,+\infty)).
\end{equation}

Our motivation for relation (\ref{f04}) arose from the desire to generalize the 
penalization result obtained for Bessel processes in Roynette, Vallois and
Yor \cite{rvyV} (see also \cite{rvyCR} and \cite{rvyI}). From our point of view, since many of the
penalization results are derived for Brownian motion and Bessel
processes, it is important to increase
understanding of the assumptions needed to guarantee the validity of such
results for more general diffusions. In particular, we prove that 
(see Theorem  \ref{thm62} and Example \ref{ex61})
\begin{equation}
\label{f05}
\lim_{t\to\infty}\frac{\E_0(h(L_t)\,|\,\cC_u)}{\E_0(h(L_t))}=S(X_u)h(L_u)+1-
  H(L_u)=:M^h_u \qquad \text{a.s.},
\end{equation}
where $h$ is a probability density function on $\R_+$  (with some nice
properties) and $H$ is the corresponding distribution function. 
\vskip.4cm
\noindent
{\bf 3.} The paper is organised as follows. In the next section basic
properties on subexponentiality are presented and a new result 
(Lemma \ref{prop0}) on the limiting behavior of the convolution of an
subexponential and a more general distribution is derived. In Section
3 we study the spectral representations of the hitting time
distributions and the L\'evy measure. In Section 4 results on
subexponentiality and the spectral representations are combined to
yield relation (\ref{f04}). Hereby we also need a weak form of a
Tauberian theorem given as Lemma \ref{lemma31} in Appendix. 
The application in penalizations is discussed in Section 5. To make
the paper more readable we state and prove first the general theorem on
penalizations. After this the penalization with local time is treated 
and (\ref{f05}) is proved. The paper is concluded by characterizing
the law of the canonical process under the penalized measure induced
by the martingale $M^h.$ Using
absolute continuity and the compensation formula for excursions 
we are able to shorten the proof when compared
with the one in \cite{rvyV}.

\section{Subexponentiality}
\label{sec10}

In this section we present some basic results on subexponential
probability distributions. Later, in Section \ref{sec3}, it is assumed
that the probability distribution induced by the tail of the L\'evy
measure of $\tau$ is subexponential. This assumption allows us to
deduce the crucial limiting behavior of the first hitting time
distribution (see Proposition \ref{prop32}). 

\begin{definition}
The probability distribution function $F$ on $(0,+\infty)$  such that 
\begin{equation}
\label{ee040}
F(0+)=0,\quad  F(x)<1\quad\forall x>0,\quad \lim_{x\to\infty}F(x)=1
\end{equation}
is called subexponential if 
\begin{equation}
\label{ee04}
\lim_{x\to +\infty}\ol{F*F}(x)\,/\,\overline F(x)=2
\end{equation}
where $*$ denotes the convolution and $\overline F(x):=1-F(x)$ the complementary distribution function.
\end{definition}

For the following two lemmas and their proofs we refer Chistyakov
\cite{chistyakov64} and 
Embrechts et al. \cite{embgolver79}. 

\begin{lemma}
\label{slovar}
If $F$ is a probability distribution function satisfying (\ref{ee040})
and  
$$
\ol F(x)\,\mathop{\sim}_{x\to \infty}\,x^{-\alpha}\, H(x)
$$ 
with $\alpha\geq 0$ and $H$ a slowly varying function  then $F$ is subexponential. 
\end{lemma}
\begin{lemma}
\label{uni}
If $F$ is subexponential then 
\begin{description}
\item{(i)}\ uniformly on compact $y$-sets 
\begin{equation}
\label{ee05}
\lim_{x\to\infty} {\ol F(x+y)}/{\ol F(x)}=1,
\end{equation}
\item{(ii)}\ for all $\ep>0,$\, 
\begin{equation}
\label{ee051}
\lim_ {x\to+\infty}{\rm e}^{\ep\,x}\ol F(x)= +\infty
\end{equation}
\end{description}
\end{lemma}

The proof of the next lemma uses some ideas from Teugels \cite{teugels75} p. 1006. 

\begin{lemma}
\label{prop0} Let $F$ and $G$ be two probability distributions on $\R_+.$
Assume that
\begin{description}
\item{(1)}\hskip3mm $F$ is subexponential,  
\item{(2)}\hskip3mm $\lim_{x\to\infty} \ol G(x)/\ol F(x)=c>0.$
\end{description}
Then
\begin{equation}
\label{ee1}
\lim_{x\to\infty} \ol{F*G}(x)/\left(\ol G(x)+\ol F(x)\right)=1.
\end{equation}
\end{lemma}

\begin{proof}
Let $\ep\in(0,1).$ By assumption {\sl (2)} there exists $\de=\de(\ep)$ 
such that for $x>\de$
\begin{equation}
\label{ee2} 
c\,(1-\ep)\ol F(x)\leq \ol G(x)\leq c\,(1+\ep)\ol F(x).
\end{equation}
 Observe that 
\begin{eqnarray*}
&&
\hskip-1cm
\ol{F*G}(x)=1-{F*G}(x)=1-F(x)+F(x)-\int_0^x G(x-y)\,dF(y)
\\
&&
\hskip3.6cm
=\ol F(x)+\int_0^x \ol G(x-y)\,dF(y).
\end{eqnarray*}
We assume now, throughout the proof, that $x>\de$ and write
\begin{equation}
\label{ee3} 
 \frac{\ol{F*G}(x)}{\ol G(x)+\ol F(x)}= 
\frac{\ol{G}(x)}{\ol G(x)+\ol F(x)}\left(I_1(x)+I_2(x)\right)
+\frac{\ol{F}(x)}{\ol G(x)+\ol F(x)},
\end{equation}
where
$$
I_1(x):=\int_0^{x-\delta} \frac{\ol G(x-y)}{\ol G(x)}\,dF(y)
$$
and 
$$
I_2(x):=\int_{x-\delta}^x \frac{\ol G(x-y)}{\ol G(x)}\,dF(y).
$$
Obviously, by assumption  {\sl (2)}, the claim (\ref{ee1}) follows if we show that 
\begin{equation}
\label{ee4} 
\lim_{x\to\infty} I_1(x)=1
\end{equation}
and
\begin{equation}
\label{ee5} 
\lim_{x\to\infty} I_2(x)=0.
\end{equation}
{\sl Proof of (\ref{ee5}).}\ Since $\ol G(x-y)\leq 1$ we have
\begin{eqnarray*}
&&
\hskip-1.65cm
 I_2(x)\leq \int_{x-\delta}^x \frac{dF(y)}{\ol
   G(x)}=\frac{F(x)-F(x-\de)}{\ol G(x)}
\\
&&\hskip2cm
=\frac{\ol F(x-\de)-\ol F(x)}{\ol G(x)}
\\
&&\hskip2cm
=\frac{\ol F(x)}{\ol G(x)}\left(\frac{\ol F(x-\de)}{\ol F(x)}-1\right).
\end{eqnarray*}
Using now (\ref{ee05}) and assumption {\sl (2)} yields (\ref{ee5}).

\noindent
{\sl Proof of (\ref{ee4}).}\  Since  
$$
\ol G(x-y)\geq \ol G(x)
$$  
we have 
$$
I_1(x)=\int_0^{x-\delta} \frac{\ol G(x-y)}{\ol G(x)}\,dF(y)\geq F(x-\de).
$$
Consequently, 
\begin{equation}
\label{ee60} 
\liminf_{x\to +\infty}I_1(x)\geq 1.
\end{equation}
To derive an upper estimate, notice first that
\begin{equation}
\label{ee6} 
I_1(x)\leq \frac{1+\ep}{1-\ep}\int_0^{x-\delta} \frac{\ol F(x-y)}{\ol F(x)}\,dF(y),
\end{equation}
because, from (\ref{ee2}),
$$
x>\de\quad\Rightarrow\quad \ol G(x)\geq c(1-\ep)\ol F(x)
$$
and
$$
y\leq x-\de\quad\Rightarrow\quad x-y\geq \de\quad\Rightarrow\quad  \ol G(x-y)\leq c(1+\ep)\ol F(x-y).
$$
Next we develop the integral term in (\ref{ee6}) as follows
\begin{eqnarray*}
&&
\int_0^{x-\delta} \ol F(x-y)\,dF(y)
\\
&&
\hskip2cm
=\int_0^{x-\delta}\left( 1- F(x-y)\right)\,dF(y)
\\
&&
\hskip2cm
=F(x-\de)-\int_0^{x-\delta} F(x-y)\,dF(y)
\\
&&
\hskip2cm
=F(x)-\int_0^{x-\delta} F(x-y)\,dF(y)+F(x-\de)-F(x)
\\
&&
\hskip2cm
=F(x)-\int_0^{x-\delta} F(x-y)\,dF(y)-\int_{x-\delta}^x\,dF(y)
\\
&&
\hskip2cm
=F(x)-\int_0^{x-\delta}F(x-y)\,dF(y)
\\
&&
\hskip3.5cm
-\int_{x-\delta}^xF(x-y)\,dF(y)-\int^x_{x-\delta}\ol F(x-y)\,dF(y)
\\
&&
\hskip2cm
\leq 
F(x)-\int_0^{x}F(x-y)\,dF(y).
\end{eqnarray*}
Hence, 
$$
\int_0^{x-\delta} \ol F(x-y)\,dF(y)\leq F(x)-F*F(x)
=\ol{F*F}(x)-\ol F(x).
$$
Consequently, from (\ref{ee6}),
$$
I_1(x)\leq \left(\frac{1+\ep}{1-\ep}\right)\,\left(\frac{\ol{F*F}(x)-\ol F(x)}{\ol F(x)}\right),
$$
and using (\ref{ee04}) and letting $\ep\to 0$ we obtain
$$
\limsup_ {x\to +\infty}I_1(x)\leq 1
$$
which together with (\ref{ee60}) proves (\ref{ee4}) completing the
proof of Lemma \ref{prop0}   
\end{proof}

\section{Spectral representations}
\label{sec2}

Spectral representations play a crucial role in our study of 
asympotic properties of the hitting time
distributions. In this section we recall basic properties of these
representations and derive some useful estimates. For references on
spectral theory of strings, we list \cite{kackrein74}, \cite{kasahara75},
\cite{dymmckean76}, \cite{kent80}, \cite{kuchler80},  \cite{knight81}
, \cite{kotaniwatanabe81}, \cite{kent82}, and
\cite{kuchlersalminen89}.  

Besides the diffusion $X$ itself, it is important to study $X$ when killed at the first hitting time of 0, denoted 
$\widehat X=\{\widehat X_t :t\geq 0\}$, i.e., the diffusion with the sample paths 
\begin{equation}
\label{kill}
\widehat X_t:=
\begin{cases}
X_t, & t<H_0,\\
\partial,& t\geq H_0,
\end{cases}
\end{equation}
where 
$
H_0:=\inf\{t:\, X_t=y\},
$
and 
$\partial$  is a point isolated from $\R_+$ (a "cemetary" point).
Then $\{\widehat X_t:t\geq 0\}$ is a diffusion    
with the same scale and speed as $X.$ 
Let $\hat p$ denote the transition density of $\widehat X$ with respect to $m:$ 
\begin{equation}
\label{kill2}
\P_x(\widehat X_t\in dy)=\P_x(X_t\in dy; t<H_0)=\hat p(t;x,y)\,m(dy).
\end{equation}
Recall that the density of the $\P_x$-distribution of $H_0$ exists and
is given by   
\begin{equation}
\label{f00}
f_{x0}(t):=\P_x(H_0\in dt)/dt=\lim_{y\downarrow 0}\frac{\hat p(t;x,y)}{S(y)}.
\end{equation}
Moreover, the L\'evy measure $\nu$ of the inverse local time $\tau,$
see (\ref{e000}) and (\ref{e00}), is absolutely continuous with respect to the Lebesgue measure, 
and the density of $\nu$ satisfies 
\begin{eqnarray}
\label{v00}
&&\hskip-.5cm
\dot\nu(v):=\nu(dv)/dv=\lim_{x\downarrow 0}\frac{f_{x0}(v)}{S(x)}
\end{eqnarray}

We define now the basic eigenfunctions $A(x;\gamma)$ and $C(x;\gamma)$
associated with $X$ and $\widehat X,$ respectively, 
via the integral equations
(recall that $S$ is continuous and $m$ has no atoms) 
$$
A(x;\gamma)=1-\gamma\int_0^xdS(y)\,\int_0^ym(dz)\, A(z;\gamma),
$$
\begin{equation}
\label{e201}
C(x;\gamma)=S(x)-\gamma\int_0^xdS(y)\,\int_0^ym(dz)\, C(z;\gamma),
\end{equation}
and the initial values
\begin{equation}
\label{e2015}
A(0;\gamma)=1,\quad A'(0;\gamma):=\lim_{x\downarrow 0}\frac{A(x;\gamma)-1}{S(x)}=0,
\end{equation}
\begin{equation}
\label{e2016}
C(0;\gamma)=0,\quad C\,'(0;\gamma):=\lim_{x\downarrow 0}\frac{C(x;\gamma)}{S(x)}=1.
\end{equation}
Let $\{A_n\}$
and  $\{C_n\}$ be two families of functions defined by 
\begin{equation}
\label{e210}
A_0(x)=1,\qquad A_{n+1}(x)=\int_0^xdS(y)\,\int_0^ym(dz)\, A_n(z)
\end{equation}
and
\begin{equation}
\label{e21}
C_0(x)=S(x),\qquad C_{n+1}(x)=\int_0^xdS(y)\,\int_0^ym(dz)\, C_n(z),
\end{equation}
respectively. Then the functions $A(x;\gamma)$ and  $C(x;\gamma)$ are explicitly given by
\begin{equation}
\label{e2150}
A(x;\gamma)=\sum_{n=0}^\infty (-\gamma)^n\,A_n(x).
\end{equation}
and
\begin{equation}
\label{e215}
C(x;\gamma)=\sum_{n=0}^\infty (-\gamma)^n\,C_n(x),
\end{equation}
respectively (see  Kac and Krein \cite{kackrein74} p. 29). In the next
lemma we give an estimate which shows that the series for $C$
converges rapidly for all values on $\gamma$ and $x\geq 0$. A similar
estimate for $A$ can be 
found in Dym and McKean \cite{dymmckean76} p. 162.

\begin{lemma}
\label{lemma1}
The functions $x\mapsto C_n(x),\, x\geq 0,\, n=0,1,2,\dots,$ are positive, increasing 
and satisfy
\begin{equation}
\label{e22}
 C_n(x)\leq \frac 1{n!}\, S(x) \left(\int_0^x M(y)\,dS(y)\right)^n
\end{equation}
where $M(z)=m(0,z).$ 
\end{lemma}
\begin{proof} The fact that $C_n$ are positive and increasing is immediate from (\ref{e21}). 
Clearly (\ref{e22}) holds for $n=0.$ Hence, consider
\begin{eqnarray*}
&&
C_{n+1}(x)=\int_0^xdS(y)\,\int_0^ym(du)\, C_n(u)\\
&&\hskip1.55cm
\leq 
\int_0^xdS(y)\,\int_0^ym(du)\, \frac 1{n!}\, S(u) \left(\int_0^u M(z)\,dS(z)\right)^n
\\
&&\hskip1.55cm
\leq 
 \frac 1{n!}\, S(x) 
\int_0^xdS(y)\,\int_0^ym(du)\,\left(\int_0^u M(z)\,dS(z)\right)^n
\\
&&\hskip1.55cm
\leq 
 \frac 1{n!}\, S(x) 
\int_0^xdS(y)\,\left(\int_0^y M(z)\,dS(z)\right)^n M(y)
\\
&&\hskip1.55cm
=
 \frac 1{(n+1)!}\, S(x) 
\left(\int_0^x M(y)dS(y)\right)^{n+1},
\end{eqnarray*}
where we have used the facts that $x\mapsto S(x)$ is increasing and
$x\mapsto M(x)$ is positive.
\end{proof}

\begin{lemma}
\label{lemma2}
The function $x\mapsto C(x;\gamma)$ satisfies the inequality
\begin{equation}
\label{e23}
|C(x;\gamma)|\leq S(x)\,\exp\left(|\gamma|\int_0^x M(z)dS(z)\right).
\end{equation}
\end{lemma}
\begin{proof} 
This follows readily from (\ref{e215}) and (\ref{e22}).
\end{proof}

From Krein's
theory of strings it is known (see  \cite{dymmckean76} p.176, and
\cite{kackrein74,
  kuchler80, kotaniwatanabe81}) that there exists a $\sigma$-finite measure denoted
$\Delta,$ called the principal spectral measure of $X$, with the property
\begin{equation} 
\label{repk1}
   \int_0^\infty \frac{\Delta(dz)}{z+1} < \infty
\end{equation}
such that the transition density of $X$ can be represented as 
\begin{equation} 
\label{krein0}
p(t;x,y)= \int_0^\infty {\rm e}^{-\gamma t}\,A(x;\gamma)\,A(y;\gamma)\,\Delta(d\gamma).
\end{equation}
We remark that from the assumption that $m$ does not have an atom at 0
it follows (see  \cite{dymmckean76} p.192) that $\Delta([0,\infty))=\infty.$

Analogously, for the killed process $\widehat X$ there exists (see 
\cite{knight81},
\cite{kuchlersalminen89}) a $\sigma$-finite measure, denoted
$\widehat\Delta$  and called the principal spectral measure of $\widehat X,$
such that
\begin{equation} 
\label{rep}
   \int_0^\infty \frac{\widehat\Delta(dz)}{z(z+1)} < \infty,
\end{equation}
and
\begin{equation} 
\label{repdelta}
   \int_0^\infty \frac{\widehat\Delta(dz)}{z} = \infty.
\end{equation}
The transition density of $\widehat X$ can be represented as 
\begin{equation} \label{krein1}
\hat p(t;x,y)= \int_0^\infty {\rm e}^{-\gamma t}\,C(x;\gamma)\,C(y;\gamma)\,\widehat\Delta(d\gamma).
\end{equation}

The result of the next proposition can be found also in
\cite{kuchlersalminen89}. Since the proof in \cite{kuchlersalminen89}
is not complete in all details we found it worthwhile to give here a
new proof.

\begin{proposition}
\label{prop1}
(i)\ The density of the $\P_x$-distribution of the first hitting time $H_0$ has the spectral representation
\begin{equation} 
\label{krein2}
f_{x0}(t)=\int_0^\infty {\rm e}^{-\gamma t}\,C(x;\gamma)\,\widehat\Delta(d\gamma).
\end{equation}
(ii) The density of the L\'evy measure of the inverse local time at 0 has the spectral representation
\begin{equation} 
\label{krein21}
\dot\nu(t)=\int_0^\infty {\rm e}^{-\gamma t}\,\widehat\Delta(d\gamma).
\end{equation}

\end{proposition}
\begin{proof}
(i)\ Combining (\ref{f00}) and (\ref{krein1}) yields
\begin{eqnarray*}
&&
f_{x0}(t)=\lim_{y\downarrow 0}\frac{\hat p(t;x,y)}{S(y)}.
\\
&&
\hskip1.15cm
=\lim_{y\downarrow 0}
\int_0^\infty {\rm e}^{-\gamma t}\,C(x;\gamma)\,\frac{C(y;\gamma)}{S(y)}\,\widehat\Delta(d\gamma).
\end{eqnarray*}
We show that the limit can be taken inside the integral by the Lebesgue dominated convergence 
theorem. Let $t>0$ be fixed an choose $\ep$ such that 
$$
t-\int_0^\varepsilon M(z)dS(z)\geq t/2.
$$
Then, from Lemma \ref{lemma2}, for $\gamma>0$ and $0<y<\ep$ we have
$$
 {\rm e}^{-\gamma t}\,\frac{|C(y;\gamma)|}{S(y)}
\leq
\exp\left(-\gamma\left(t-\int_0^y M(z)dS(z)\right)\right)
\leq
 {\rm e}^{-\gamma t/2}
$$
Consequently, it remains to show that
\begin{equation}
\label{e24}
\int_0^\infty {\rm e}^{-\gamma t/2}\left|C(x;\gamma)\right|\widehat\Delta(d\gamma)<\infty.
\end{equation}
By the Cauchy-Schwartz inequality
\begin{eqnarray*}
&&
\left(\int_0^\infty {\rm e}^{-\gamma
  t/2}\left|C(x;\gamma)\right|\widehat\Delta(d\gamma)\right)^2
\\
&&\hskip1.5cm
\leq
\int_0^\infty {\rm e}^{-\gamma t/2}\left(C(x;\gamma)\right)^2\widehat\Delta(d\gamma)
\int_0^\infty {\rm e}^{-\gamma t/2}\widehat\Delta(d\gamma)
\\
&&\hskip1.5cm
=
\hat p(t/2;x,x)
\int_0^\infty {\rm e}^{-\gamma t/2}\widehat\Delta(d\gamma).
\end{eqnarray*}
Clearly, 
$
\hat p(t/2;x,x)<\infty
$
and, by (\ref{rep}), 
$
\int_0^\infty {\rm e}^{-\gamma t/2}\widehat\Delta(d\gamma)<\infty.
$
These estimates allow us to use the Lebesgue dominated 
convergence theorem and since (cf. (\ref{e2016}))
$$
\lim_{y\to 0}C(y;\gamma)/S(y)= C\,'(0;\gamma)=1
$$
the proof of (i) is complete. Representation 
(\ref{krein21}) can be proved similarly using formula (\ref{v00}),
(\ref{krein2}), (\ref{e2016}) and
the estimates derived above. We leave the details to the reader. 
\end{proof}

\begin{remark}
Consider
\begin{eqnarray*}
&&
\hskip-1.7cm
\int_0^\infty (1 \wedge t) \,\dot\nu(t)\,dt  =\int_0^\infty dt\, (1 \wedge
t) 
\int_0^\infty\widehat\Delta(d\gamma)\, {\rm e}^{-\gamma t}
\\
&&
\hskip1.6cm
=
\int_0^\infty\widehat\Delta(d\gamma) \int_0^\infty dt\, (1 \wedge
t)\, {\rm e}^{-\gamma t}.
\end{eqnarray*}
A straightforward integration yields
$$
\int_0^\infty (1 \wedge
t)\, {\rm e}^{-\gamma t}\,dt= \frac 1{\ga^2}\left(1-{\rm e}^{-\ga}\right),
$$
and, consequently, $(\ref{rep})$ is equivalent with 
(cf. \cite{knight81}) 
$$
\int_0^\infty (1 \wedge t) \,\dot\nu(t)\,dt<\infty,
$$
which is the crucial property of  
the L\'evy measure of a subordinator. 
For (\ref{repdelta}), see \cite{kackrein74} p. 82.
and \cite{kuchlersalminen89}.
\end{remark}

\begin{example}
\label{example1} 
Let $R=\{R_t: t\geq 0\}$ and $\widehat R=\{\widehat R_t: t\geq 0\}$ be Bessel processes of dimension $0<\delta<2$
reflected at 0 and killed at 0, respectively.  We
compute explicit spectral representations associated with $R$ and
$\widehat R.$  

From, e.g., \cite{borodinsalminen02} p. 133
 the following information concerning $R$ and   $\widehat R$ can be found:
\begin{description}
\item Speed measure
\begin{equation}
\label{m}
m(dx)=2\,x^{1-2\al}\, dx\quad \quad \al:=(2-\delta)/2.
\end{equation}
\item Scale function
\begin{equation}
\label{s}
S(x)= \frac{1}{2\al}\,x^{2\al}
\end{equation}
\item Transition density  of $R$ (w.r.t. $m$)
\begin{equation}
\label{p}
p(t;x,y)= \frac 1 {2t} (xy)^\al\exp\left(-\frac{x^2+y^2}{2t}\right)
I_{-\al}\left(\frac{xy}{t}\right),\quad x,y>0. 
\end{equation}
\item Transition density  of $\widehat R$ (w.r.t. $m$)
\begin{equation}
\label{hatp}
\hat p(t;x,y)= \frac 1 {2t} (xy)^\al\exp\left(-\frac{x^2+y^2}{2t}\right)
I_{\al}\left(\frac{xy}{t}\right),\quad x,y>0.  
\end{equation}
\end{description}
To find the Krein measure $\Delta$ associated with $R$ we exploit
formulas (\ref{krein0}) and (\ref{p}) with $x=y=0$ and use 
$$
I_\nu(z)\,\sim\,\frac 1{\Gamma(\nu+1)}\left(\frac z2\right)^\nu,\quad
z\to 0
$$
to obtain 
$$
p(t;0,0)=\lim_{x,y\to 0}p(t;x,y)=
\frac{t^{-(1-\al)}}{2^{1-\al}\,\Gamma(1-\al)}=\int_0^\infty {\rm e}^{-\gamma
  t}\,\Delta(d\ga).
$$
Inverting the Laplace transform yields
\begin{equation}
\label{e241}
\Delta(d\ga)=\frac{\gamma^{-\al}\,d\ga}{2^{1-\al}\,(\Gamma(1-\al))^2}.
\end{equation}
We apply formula (\ref{e2150}), (\ref{m}), and (\ref{s}) to find the function $A(x;\ga),$
and, hence, compute first directly via (\ref{e210}) 
$$
A_n(x)=\frac{\Gamma(1-\al)\,x^{2n}}{2^n\,\Gamma(n+1)\,\Gamma(n+1-\al)},\quad n=0,1,2,\dots.
$$
Consequently, after some manipulations, we have 
$$
A(x;\ga)=\Gamma(1-\al)\,2^{-\al}\,\left(x\sqrt{2\ga}\right)^\al\,J_
       {-\al}\left(x\sqrt{2\ga}\right),
$$
where $J$ denotes the usual Bessel function of the first kind, i.e., 
$$
J_\nu(z)=\sum_{n=0}^\infty \frac{(-1)^n(z/2)^{\nu+2n}}{\Gamma(n+1)\,\Gamma(\nu+ n+1)},
$$
and, finally, putting pieces together into (\ref{krein0}) yields
\begin{equation}
\label{e2411}
p(t;x,y)=\frac 12\,\int_0^\infty {\rm e}^{-\ga t}\, (xy)^{\al}\,J_
       {-\al}\left(x\sqrt{2\ga}\right)\,J_
       {-\al}\left(y\sqrt{2\ga}\right)\, d\ga.
\end{equation}
Next we compute the Krein measure $\widehat\Delta$ associated with
$\widehat R.$ For this, we deduce from (\ref{f00}), (\ref{v00}),
(\ref{s}), and (\ref{hatp})
\begin{equation}
\label{e24115}
\dot\nu(t)=\lim_{x,y\to 0}\frac{\hat p(t;x,y)}{S(x)S(y)}
=\frac{2^{1-\al}\, \al\, t^{-(1+\al)}}{\Gamma(\al)}=\int_0^\infty {\rm e}^{-\gamma
  t}\,\widehat\Delta(d\ga),
\end{equation}
and, consequently, inverting the Laplace transform gives
\begin{equation}
\label{e2412}
\widehat \Delta(d\ga)=\frac{2^{1-\al}\, \ga^{\al}}{(\Gamma(\al)) ^2}\, d\ga
\end{equation}
Similarly as above, we apply formula (\ref{e215}) to find the function $C(x;\ga),$
and, hence, compute first directly via (\ref{e21}) 
$$
C_n(x)=\frac{\Gamma(\al)\,x^{2\al+2n}}{2^{n+1}\,\Gamma(n+1)\,\Gamma(n+1+\al)},\quad n=0,1,2,\dots.
$$
Consequently, after some manipulations,
$$
C(x;\ga)=\Gamma(\al)\,2^{(\al-2)/2}\,\ga^{-\al/2}\,x^\al\, J_{\al}\left(x\sqrt{2\ga}\right).
$$
and
\begin{equation}
\label{e2413}
\hat p(t;x,y)=\frac 12\,\int_0^\infty {\rm e}^{-\ga t}\, (xy)^{\al}\,J_
       {\al}\left(x\sqrt{2\ga}\right)\,J_
       {\al}\left(y\sqrt{2\ga}\right)\, d\ga.
\end{equation}
See also Karlin and Taylor \cite{karlintaylor81} p. 338.
\end{example}
\begin{example}
\label{example11} Taking above $\alpha=1/2$ yields formulas for
Brownian motion. Recall 
$$
J_{1/2}(z)=\sqrt{\frac 2{\pi z}}\,\sin{z},
\quad \text{and}\quad 
J_{-1/2}(z)=\sqrt{\frac 2{\pi z}}\,\cos{z}.
$$ 
Consequently, from (\ref{e2411})
\begin{eqnarray}
\label{ex1}
&&
p(t;x,y)=\frac 1{\pi}\int_0^\infty {\rm e}^{-\gamma\,t}\cos(x\sqrt{2\ga})\cos(y\sqrt{2\ga})\, \frac{d\gamma}{\sqrt{2\ga}}
\\
&&
\nonumber
\hskip1.6cm
=\frac{1}{2\sqrt{2\pi t}}\left({\rm e}^{-(x-y)^2/(2t)}+{\rm e}^{-(x+y)^2/(2t)}\right),
\end{eqnarray}
and from  (\ref{e2413})   
\begin{eqnarray*}
&&
\hat p(t;x,y)=\frac 1\pi\int_0^\infty {\rm e}^{-\gamma
  t}\,\frac{\sin(x\sqrt{2\ga})}{\sqrt{2\ga}}\,
\frac{\sin(y\sqrt{2\ga})}{\sqrt{2\ga}}\,{\sqrt{2\ga}}\,d\ga
\\
&&
\hskip1.6cm
=\frac{1}{2\sqrt{2\pi t}}\left({\rm e}^{-(x-y)^2/(2t)}-{\rm e}^{-(x+y)^2/(2t)}\right).
\end{eqnarray*}
Moreover,
\begin{eqnarray*}
&&
f_{x0}(t)
=
\frac 1\pi\int_0^\infty\,{\rm e}^{-\gamma t}\,\sin(x\sqrt{2\ga})\, d\gamma
=\frac{x}{t^{\,3/2}\sqrt{2\pi}}\,{\rm e}^{-x^2/(2t)},
\end{eqnarray*}
and
\begin{eqnarray}
\label{ex2}
&&
\dot\nu(t)=
\frac 1\pi\int_0^\infty {\rm e}^{-\gamma t} \,\sqrt{2\ga}\,d\gamma
=\frac1{t^{\,3/2}\sqrt{2\pi}}.
\end{eqnarray}
From (\ref{ex1}) we obtain  $\Delta(d\ga)=d\ga/(\pi\sqrt{2\ga}),$ and  
from (\ref{ex2}) $\widehat\Delta(d\ga)=\sqrt{2\ga}\,d\ga/\pi.$
See also Karlin and Taylor \cite{karlintaylor81} p. 337 and 393, and \cite{borodinsalminen02} p. 120.
\end{example}

\begin{proposition}
\label{prop2}
(i)\ The complementary $\P_x$-distribution function of $H_0$ has the spectral representation
\begin{equation} 
\label{krein4}
\P_x(H_0>t)=\int_0^\infty \frac 1\gamma\, {\rm e}^{-\gamma t}\, \,C(x;\gamma)\,\widehat\Delta(d\gamma).
\end{equation}
(ii)\ The L\'evy measure has the spectral representation
\begin{equation} 
\label{krein41}
\nu((t,\infty))=\int_t^{\infty} \dot\nu(s)\,ds =\int_0^\infty \frac 1\gamma\, {\rm e}^{-\gamma t}\, \widehat\Delta(d\gamma).
\end{equation}

\end{proposition}
\begin{proof}
Formulas (\ref{krein4}) and (\ref{krein41}) follow from (\ref{krein2})
and (\ref{krein21}), respectively, using Fubini's theorem. To obtain
(\ref{krein41}) is straightforward but for (\ref{krein4})
the
applicability of Fubini's theorem needs to be justified.
Indeed, from (\ref{krein2}) we have informally
\begin{eqnarray*}
&&
\P_x(H_0>t)=\int_t^\infty f_{x0}(s)\,ds =\int_t^\infty
ds\int_0^\infty\widehat\Delta(d\gamma)\, {\rm e}^{-\gamma s}\,C(x;\gamma)
\\
&&
\hskip2.1cm
=\int_0^\infty\widehat\Delta(d\gamma)\int_t^\infty
ds\, {\rm e}^{-\gamma s}\,C(x;\gamma)
\end{eqnarray*}
leading to (\ref{krein4}). To make this rigorous, we verify that for
all $x>0$ 
$$
\int_0^\infty \frac 1\gamma\, {\rm e}^{-\gamma t} \,|C(x;\gamma)|\,\widehat\Delta(d\gamma)<\infty.
$$
Consider first for $\varepsilon>0$ 
$$
K_1:=\int_0^\varepsilon  \frac 1\gamma\, {\rm e}^{-\gamma t}\,|C(x;\gamma)|\,\widehat\Delta(d\gamma).
$$
By the basic estimate (\ref{e23}) for $0<\gamma<\ep$
$$
|C(x;\gamma)|\leq S(x)\,\exp\left(\ep\,\int_0^x M(z)dS(z)\right).
$$
and, consequently,
$$
K_1\leq S(x)\,\exp\left(\ep\,\int_0^x M(z)dS(z)\right)
\int_0^\infty \frac 1\gamma\, {\rm e}^{-\gamma
  t}\,\widehat\Delta(d\gamma)<\infty
$$
by (\ref{rep}). Next, let
$$
K_2:=\int_\ep^\infty  \frac 1\gamma\, {\rm e}^{-\gamma t}\,|C(x;\gamma)|\,\widehat\Delta(d\gamma).
$$
By the Cauchy-Schwartz inequality
$$
K_2^2\leq \int_\ep^\infty  \gamma^{-2} \, {\rm e}^{-\gamma
  t}\,\widehat\Delta(d\gamma)\ 
\int_\ep^\infty  \, {\rm e}^{-\gamma t}\,(C(x;\gamma))^2\,\widehat\Delta(d\gamma).
$$
The first term on the right hand side is finite by (\ref{rep}). For
the second term we have 
\begin{eqnarray*}
&&
\int_\ep^\infty  \, {\rm e}^{-\gamma t}\,(C(x;\gamma))^2\,\widehat\Delta(d\gamma)
\leq 
\int_0^\infty  \, {\rm e}^{-\gamma t}\,(C(x;\gamma))^2\,\widehat\Delta(d\gamma).
\\
&&
\hskip4.7cm
\leq \hat p(t;x,x)<\infty.
\end{eqnarray*}
The proof of (\ref{krein4}) is now complete. 
\end{proof}

\section{Asymtotic behavior of the distribution of $L_t$ as $t\to+\infty$}
\label{sec3}

We make the following assumption concerning the 
L\'evy measure of the inverse local time process 
$\{\tau_\ell\,:\,\ell\geq 0\}$ valid throughout the rest of the  paper (if nothing else is stated)
\begin{description}
\item{(A)}\hskip3mm {\sl The probability distribution function
$$   
 x\mapsto \frac{\nu(1,x]}{\nu(1,+\infty)},\quad x>1,
$$
is assumed to be subexponential.}
\end{description}

It is known, see Sato \cite{sato99} p. 164, that Assumption (A) is
equivalent with  
\begin{equation}
\label{3e00}
\P(\tau_\ell\geq t) \,\mathop{\sim}_{t\to +\infty}\,
\ell\,\nu((t,+\infty))\quad  \forall\ \ell>0,
\end{equation}
and also with
\begin{equation}
\label{3e000}
{\text{\sl The\ law\ of\ }} \tau_\ell\ {\text{\sl
    is\ subexponential\ for\ every\ }} \ell>0.
\end{equation}

\begin{proposition}
\label{prop31}
For any fixed $\ell>0$, it holds 
\begin{equation}
\label{3e0}
 \P_0(L_t\leq \ell) \,\mathop{\sim}_{t\to +\infty}\, \ell\,\nu((t,+\infty)).
\end{equation}
\end{proposition}

\begin{proof}
The claim follows immediately from (\ref{3e00}) since 
$$
\P_0(L_t\leq \ell)= \P(\tau_\ell\geq t).
$$
\end{proof}

Our goal is to study the asymptotic behavior of $L_t$ under $\P_x.$ 
For this, we analyze first the distribution of the hitting
time $H_0.$ The proof of the next proposition is based on Lemma
\ref{lemma31} stated and proved in Section \ref{20} below.

\begin{proposition}
\label{prop32}
For any $x> 0,$ it holds 
\begin{equation}
\label{3e1}
 \P_x(H_0>t) \,\mathop{\sim}_{t\to +\infty}\, S(x)\,\nu((t,+\infty)).
\end{equation}
\end{proposition}

\begin{proof}
Recall from  (\ref{krein4}) and (\ref{krein41}) in Proposition
\ref{prop2} the spectral representations 
\begin{equation} 
\label{krein4n}
\P_x(H_0>t)=\int_0^\infty \frac 1\gamma\, {\rm e}^{-\gamma t}\, \,C(x;\gamma)\,\widehat\Delta(d\gamma)
\end{equation}
and 
\begin{equation} 
\label{krein41n}
\nu((t,+\infty))=\int_0^\infty \frac 1\gamma\, {\rm e}^{-\gamma t}\, \widehat\Delta(d\gamma).
\end{equation}
We apply Lemma \ref{lemma31} with $\mu(d\ga)=\widehat
\Delta(d\ga)/\ga,$ $g_1(\ga)=C(x;\gamma)$ and $g_2(\ga)=S(x).$ Then,
the mapping $t\mapsto \P_x(H_0>t)$ has the r\^ole of 
$f_1$ and $t\mapsto S(x)\,\nu((t,+\infty))$ the r\^ole of $f_2.$ 
Condition (\ref{c1}) takes the form
$$
\lim_ {t\to\infty}S(x)\,\nu((t,+\infty))\,{\rm e}^{bt}=0
$$
and this holds by Assumption (A) and (\ref{ee051}).
Moreover, condition (\ref{c2}) means now 
$$
\lim_{\gamma\to 0}C(x;\gamma)/S(x)=1  
$$
and this is true since using estimate (\ref{e22}) in (\ref{e215}) we obtain
$$
\left|\frac{C(x;\gamma)}{S(x)}-1\right|\leq \al\,|\gamma| {\rm e}^{\be\,|\gamma|}   
$$
with some $\al$ and $\be$ depending only on $x.$
Consequently, (\ref{c3}) in
Lemma \ref{lemma31}  holds and, hence, the proof of the proposition is complete.
\end{proof}

The main result of this section is as follows.

\begin{proposition}
\label{prop33}
For any $x> 0$ and $\ell>0,$  it holds 
\begin{equation}
\label{3e2}
 \P_x(L_t\leq \ell) \,\mathop{\sim}_{t\to +\infty}\, (S(x)+\ell)\,\nu((t,+\infty)).
\end{equation}
\end{proposition}

\begin{proof} 
Since $L_t$ increases only when $X$ is at 0 we may write
\begin{eqnarray*}
&& 
\P_x(L_t\leq \ell) =  \P_x(H_0>t) + \P_x(H_0< t\,,\, L_t\leq \ell)
\\
&&
\hskip2.1cm
=
\P_x(H_0>t) + \P_x(H_0< t\,,\, L_{t-H_0}\circ\te_{H_0}\leq \ell)
\\
&&
\hskip2.1cm
= 
\P_x(H_0>t) + \P_x(H_0< t\,,\, t-H_0\leq \hat \tau_\ell),
\end{eqnarray*}
where $\te_\cdot$ denotes the usual shift operator and  $\hat
\tau_\ell$ is a subordinator starting from 0, independent of $H_0$ and
identical in law with $\tau_\ell$ (under $\P_0$), by the strong Markov property.   
Consequently, 
\begin{eqnarray*}
&&
\hskip-.5cm
\P_x(L_t\leq \ell)=\P_x(H_0>t) +  \P_x( \hat\tau_\ell+H_0\geq t) -\P_x( \hat\tau_\ell+H_0\geq t\,,\, H_0> t)
\\
&&
\hskip1.6cm
= 
\P_x( \hat\tau_\ell+H_0\geq t).
\end{eqnarray*}
We use Lemma \ref{prop0} and take therein $F$ to be the $P_x$-distribution
$\hat\tau_\ell$ (which is the same as the $P_0$-distribution
$\tau_\ell$) and $G$ the $P_x$-distribution of $H_0.$ Then, by
(\ref{3e000}), $F$ is subexponential and from (\ref{3e0}) and
(\ref{3e1}) we have 
$$
\lim_{t\to\infty}\frac{\P_x(H_0> t)}{\P_x( \hat\tau_\ell> t)}=\frac{S(x)}\ell>0. 
$$
Consequently, by Lemma \ref{prop0},
$$
\lim_{t\to\infty}\frac{\P_x( \hat\tau_\ell+H_0> t)}{\P_x(
  \hat\tau_\ell> t)+\P_x( H_0> t)}
=1,
$$
in other words, 
\begin{eqnarray*}
&&
\P_x(L_t\leq \ell)
\,\mathop{\sim}_{t\to \infty}\, 
\P_x( H_0> t)+\P_x(
  \hat\tau_\ell> t)
\\
&&
\hskip2cm
\,\mathop{\sim}_{t\to \infty}\, 
S(x)\,\nu((t,\infty))+\ell\,\nu((t,\infty)),
\end{eqnarray*}
as claimed.
\end{proof}

\begin{example}
\label{bessel1}
{\rm For a Bessel process of dimension $d\in(0,2)$ reflected at 0 
we have from (\ref{e24115}) in  Example \ref{example1}
$$
\nu((t,+\infty))=\frac{2^{1-\al}}{\Gamma(\al)}\,t^{-\al},
$$
and  Assumption (A) holds by Lemma \ref{slovar}. Consequently,
$$
\P_x(L_t<\ell) 
\,\mathop{\sim}_{t\to \infty}\, 
(S(x)+\ell)\,\nu((t,+\infty)).
$$
where the scale function is as in Example \ref{example1}. Taking here
$\al=1/2$ gives formulae for reflecting Brownian motion. We remark
that our normalization of the local time (see (\ref{e000})) is different from the one
used in  Roynette
et al. \cite{rvyV} Section 2. In our case,  from (\ref{e1}) and
(\ref{e24115}) it follows (cf. also   \cite{borodinsalminen02} p. 133
where the resolvent kernel is explicitly given) that
\begin{equation}
\label{cf1}
\E_0\left(\exp(-\lambda
\tau_\ell)\right)=\exp\left(-\ell\, \frac{\Gamma(1-\al)}{\Gamma(\al)}\,
2^{1-\al}\, \lambda^{\al}\right).
\end{equation}
Comparing now formula (2.11) in
\cite{rvyV} with (\ref{cf1}) it is seen that 
$$
\widehat L_t = 2\al\, L_t
$$
where $\widehat L$ denotes the local time used in \cite{rvyV}.
}

\end{example}


\section{Penalization of the diffusion with its local time}
\label{sec6}

\subsection{General theorem of penalization}
\label{sec61} 
Recall that $(C,\cC,\{\cC_t\})$ denotes the canonical space of continuous functions, and let $\P$ be a probability measure defined therein. In the next theorem we present the general penalization result which we then specialize to the penalization with local time.

\begin{theorem}
\label{thm61}
Let $\{F_t: t\geq 0\}$ be a stochastic process (so called weight process) satisfying
$$
0<\E(F_t)<\infty\quad \forall\ t>0.
$$    
Suppose that for any $u\geq 0$ 
\begin{equation}
\label{61}
\lim_{t\to\infty}\frac{\E(F_t\,|\,\cC_u)}{\E(F_t)}=:M_u
\end{equation}
exists a.s. and 
\begin{equation}
\label{62}
\E(M_u)=1.
\end{equation}
Then 
\begin{description}
\item{1)} $M=\{M_u: u\geq 0\}$ is a non-negative martingale with $M_0=1,$
\item{2)} for any $u\geq 0$ and $\Lambda\in\cC_u$ 
\begin{equation}
\label{63}
\lim_{t\to\infty}\frac{\E({\bf 1}_{\Lambda}\,F_t)}{\E(F_t)}=\E({\bf 1}_{\Lambda}\,M_u)=:\Q^{(u)}(\Lambda),
\end{equation}
\item{3)} there exits  a probability measure $\Q$ on $(C,\cC)$ such
  that for any $u>0$ 
$$
\Q(\Lambda)=\Q^{(u)}(\Lambda)\qquad \forall \Lambda\in\cC_u.  
$$ 
\end{description}
\end{theorem}
\begin{proof}  We have (cf. Roynette et al. \cite{rvyII})
$$
\frac{\E({\bf 1}_{\Lambda_u}\,F_t)}{\E(F_t)}=\E\left({\bf 1}_{\Lambda_u}\,\frac{\E(F_t\,|\,\cC_u)}{\E(F_t)}\right),
$$
and by (\ref{61}) and (\ref{62}) the family of random variables 
$$
\left\{\frac{\E(F_t\,|\,\cC_u)}{\E(F_t)}\,:\, t\geq 0\right\}
$$
is uniformly integrable by Sheffe's lemma (see, e.g., Meyer
\cite{meyer66}), 
and, hence, (\ref{63}) holds in ${\bf L}^1(\Omega).$ To verify the
martingale property of $M$ notice that if $u<v$ then
$\Lambda_u\in\cC_v$ and by  (\ref{63}) we have also 
$$
\lim_{t\to\infty}\frac{\E({\bf 1}_{\Lambda_u}\,F_t)}{\E(F_t)}
=\E({\bf 1}_{\Lambda_u}\,M_v).
$$
Consequently,
$$
\E({\bf 1}_{\Lambda_u}\,M_v) =\E({\bf 1}_{\Lambda_u}\,M_u),
$$
i.e., $M$ is a martingale. Since the family $\{\Q^{(u)}: u\geq 0\}$ of
probability measures is consistent, claim 3) follows from Kolmogorov's
existence theorem (see, e.g., Billingsley \cite{billingsley68} p. 228-230).
\end{proof}

\subsection{Penalization with local time}
\label{sec62}

We are interested in analyzing the penalizations of diffusion $X$ with the weight process 
given by
\begin{equation}
\label{635}
F_t:= h(L_t),\quad  t\geq 0
\end{equation}
with a suitable function $h.$ 
In particular, if $h={\bf 1}_{[0,\ell)}$ for some fixed $\ell>0$ then
  $F_t={\bf 1}_{\{L_t<\ell\}}.$ In the next theorem we prove under some
assumtions on $h$ the validity of the basic penalization hypotheses
(\ref{61}) and (\ref{62}) 
for the weight process $\{F_t: t\geq 0\}.$ The explicit form of the
corresponding martingale $M^h$ is given. In Section 6.3 it is
seen that $M^h$ remains to be a martingale for more general functions
$h,$ and properties of $X$ under the probability measure induced by
$M^h$ are discussed.

In Roynette et al. \cite{rvyV} this kind of penalizations via local times of Bessel processes with dimension
parameter $d\in(0,2)$ are studied. Our work generalizes Theorem 1.5 in
\cite{rvyV} for diffusions 
with subexponential L\'evy measure.

\begin{theorem}
\label{thm62}
Let $h:[0,\infty)\mapsto [0,\infty)$ be a Borel measurable,
    right-continuous and non-increasing 
function with compact support in $[0,K]$ for some given  $K>0.$ Assume
also that
$$
\int_0^K h(y)\,dy =1,
$$
and define for $x\geq 0$
$$
H(x):=\int_0^x h(y)\,dy. 
$$
Then for any $u\geq 0$ 
\begin{equation}
\label{64}
\lim_{t\to\infty}\frac{\E_0(h(L_t)\,|\,\cC_u)}{\E_0(h(L_t))}=S(X_u)h(L_u)+1-
  H(L_u)=:M^h_u \qquad \text{a.s.}
 \end{equation}
and 
\begin{equation}
\label{645} 
\E_0\left(M^h_u\right)=1.
\end{equation}
Consequently, statements 1), 2) and 3) in Theorem \ref{thm61} hold.
\end{theorem}

\begin{proof}
{\sl\ I)} We prove first (\ref{64}). 
\hfill\break\hfill
{\sl a)} To begin with, the following result on
the behavior of the denominator in (\ref{64}) is needed:
for any $a\geq 0$ 
\begin{equation}
\label{65}
\E_a(h(L_t))\,\mathop{\sim}_{t\to +\infty}\,
\left(S(a)h(0)+1\right)\,\nu((t,\infty)).
\end{equation}
To show this, let $\mu$ denote the measure induced by $h,$ i.e., $\mu(dy)=-dh(y).$ Then
\begin{equation}
\label{636}
h(x)=\int_{(x,K]}\mu(dy)=\int_{(0,K]}{\bf 1}_{\{y>x\}}\mu(dy),
\end{equation}
and, consequently, 
$$
\E_a(h(L_t))=\E_a\left( \int_{(0,K]}{\bf 1}_{\{\ell>L_t\}}\mu(d\ell)\right)
= \int_{(0,K]}\P_a(L_t<\ell)\mu(d\ell).
$$
By Proposition \ref{prop33} 
$$
\lim_{t\to\infty}\frac{\P_a(L_t<\ell)}{\nu((t,\infty))}=S(a)+\ell.
$$
Moreover, for $\ell\leq K$ 
$$
\frac{\P_a(L_t<\ell)}{\nu((t,\infty))}\leq
\frac{\P_a(L_t<K)}{\nu((t,\infty))}\to S(a)+K\quad \text{as}\ t\to\infty, 
$$
and, by the dominated convergence theorem,
$$
\lim_{t\to\infty}\int_{(0,K]}\frac{\P_a(L_t<\ell)}{\nu((t,\infty))}\,\mu(d\ell)
=\int_{(0,K]}(S(a)+\ell)\,\mu(d\ell).
$$ 
Hence,
\begin{equation}
\label{66}
\E_a(h(L_t))\,\mathop{\sim}_{t\to +\infty}\,\left(\int_{(0,K]}(S(a)+\ell)\,\mu(d\ell)\right) \,\nu((t,\infty)), 
\end{equation}
and the integral in (\ref{66}) can be evaluated as follows:
\begin{eqnarray*}
&&
\hskip-1cm
\int_{(0,K]}(S(a)+\ell)\,\mu(d\ell)=S(a)\int_{(0,K]}\,\mu(d\ell)+\int_{(0,K]}\ell\,\mu(d\ell)
\\
&&
\hskip3cm
=S(a)h(0)+\int_{(0,K]}\mu(d\ell)\int_0^\ell du
\\
&&
\hskip3cm
=S(a)h(0)+\int_0^K du\int_{(u,K)}\mu(d\ell)
\\
&&
\hskip3cm
=S(a)h(0)+\int_0^K h(u) du
\\
&&
\hskip3cm
=S(a)h(0)+1.
\end{eqnarray*}
This concludes the proof of (\ref{65}). 
\hfill\break\hfill
{\sl b)} To proceed with the proof of (\ref{64}),
recall that $\{L_s\,:\,s\geq 0\}$ is an additive functional, that is, 
$L_t=L_u+L_{t-u}\circ \te_u$ for $t>u$ where $\te_u$ is the usual shift operator. 
Hence, by the Markov property, for $t>u$ 
\begin{equation}
\label{67}
\E_0(h(L_t)\,|\,\cC_u)=\E_0(h(L_u+L_{t-u}\circ \te_u)\,|\,\cC_u)=H(X_u,L_u,t-u)
\end{equation}
with
$$
H(a,\ell,r):=\E_a(h(\ell+L_r)).
$$
By (\ref{65}), since $x\mapsto h(\ell+x)$ is non-increasing with compact support,
$$
H(a,\ell,r)\,\mathop{\sim}_{t\to +\infty}\, \left(S(a)h(\ell)+\int_0^\infty h(\ell+u)du\right)\,\nu((t,\infty)). 
$$
Bringing together (\ref{67}) and (\ref{65})
 with $a=0$ yields
$$
\lim_{t\to\infty}\frac{\E_0(h(L_t)\,|\,\cC_u)}{\E_0(h(L_t))}=\frac{S(X_u)h(L_u)+\int_{L_u}^\infty h(x)dx}{\int_{0}^\infty h(x)dx}
$$
completing the proof of  (\ref{64}).
\\
\noindent
{\sl II)} To verify (\ref{645}), we show that  
$\{M^h_t\,:\, t\geq 0\}$ defined in (\ref{64}) is a
non-negative martingale with $M^h_0=1$ (cf. Theorem \ref{thm61}
statement {\sl 1)}). 
\hfill\break\hfill
{\sl a)} First, consider the process  $S(X)=\{S(X_t)\,:\,t\geq
0\}.$ Since the scale function is increasing $S(X)$ 
is a
non-negative linear diffusion. Moreover, e.g., from  Meleard 
\cite{meleard86}, it is, in fact, a sub-martingale with the
Doob-Meyer decomposition    
\begin{equation}
\label{701}
S(X_t)= \tilde Y_t+\tilde L_t,
\end{equation}
where  $\tilde Y$ is a martingale and $\tilde L$ is a non-decreasing adapted
process. Because $\tilde L$ increases only when $S(X)$ is at 0 or,
equivalently, $X$ is at 0 $\tilde L$ is a local time of $X$ at
0. Consequently, $\tilde L$ is a multiple of $L$ a.s. (for the
normalization of $L,$ see (\ref{e000})), i.e., there is a non-random constant $c$ such that for
all $t\geq 0$
\begin{equation}
\label{7012}
\tilde L_t= c\, L_t.
\end{equation}
We claim that $\tilde L$ coincides with $L,$ that is $c=1.$ 
To show this, recall that 
$$
\E_x(L^{(y)}_t)=\int_0^t p(s;x,y)\, ds,
$$ 
which yields (cf. (\ref{a0}))
$$
R_ \lambda(0,0)=\int_0^\infty \lambda\,{\rm e}^{-\lambda t}\,
\E_0(L_t)\,dt.
$$
From (\ref{701}) and (\ref{7012}) we have  $\E_0(L_t)={\displaystyle \frac 1c\,\E_0(S(X_t))},$ and, hence,
\begin{eqnarray}
\label{n70}
&&
\nonumber
R_ \lambda(0,0)=\frac 1 c \int_0^\infty \lambda\,{\rm e}^{-\lambda t}\,
\E_0(S(X_t))\,dt 
\\
&&
\hskip1.6cm
= 
\frac 1 c \int_0^\infty S(y)\,\lambda\,
R_ \lambda(0,y)\, m(dy).
\end{eqnarray}
Next recall that the resolvent kernel can be expressed as 
\begin{equation}
\label{f000}
R_ \lambda(x,y)= w^{-1}_\lambda\, \psi_\lambda(x)\,
\varphi_\lambda(y),\quad 0\leq x\leq y,
\end{equation}
where  $w_\lambda$ is a constant (Wronskian) and $\varphi_\lambda$ and
$\psi_\lambda$ are the
fundamental decreasing and increasing, respectively, solutions of the generalized differential
equation
\begin{equation}
\label{f001}
\frac d{dm}\frac d{dS}\, u= \lambda u
\end{equation}
characterized (probabilistically) by 
\begin{equation}
\label{f002}
\E_x\left({\rm e}^{-\lambda H_y}\right)=\frac{R_ \lambda(x,y)}{R_
  \lambda(y,y)}.
\end{equation}
Consequently, (\ref{n70}) is equivalent with 
\begin{eqnarray}
\label{n701}
&&
\nonumber
\varphi_\lambda(0)= 
\frac 1 c \int_0^\infty S(y)\,\lambda\,
\varphi_ \lambda(y)\, m(dy)
\\
&&
\nonumber
\hskip1.2cm
=
\frac 1 c \int_0^\infty S(y)\,\frac d{dm}\frac d{dS}\, \varphi_ \lambda(y)\, m(dy).
\\
\nonumber
&&
\hskip1.2cm
=
\frac 1 c \int_0^\infty dS(y)\,\int_y^\infty m(dz)\, \frac d{dm}\frac d{dS}\, \varphi_ \lambda(z).
\\
&&
\hskip1.2cm
=
\frac 1 c \int_0^\infty dS(y)\lp\frac d{dS}
\varphi_ \lambda(+\infty)-\frac d{dS}\, \varphi_ \lambda(y)\rp,
\end{eqnarray}
where for the third equality we have used Fubini's theorem. Next we claim that 
\begin{equation}
\label{n71}
\frac d{dS}\,\varphi_ \lambda(+\infty):=\lim_{x\to\infty} \frac d{dS}\,\varphi_ \lambda(x)=0.
\end{equation}
To prove this, recall that the Wronskian (a constant) is given for all $z\geq 0$ by
\begin{equation}
\label{n72}
w_\lambda=\varphi_\lambda(z)\,\frac d{dS}\,
\psi_ \lambda(z)+\psi_\lambda(z)\,\lp -\frac d{dS}\,\varphi_ \lambda(z)\rp.
\end{equation}
Notice that both terms on the right hand side are non-negative. Since the boundary point $+\infty$ is
assumed to be natural it holds that $\lim_{z\to\infty} H_z=+\infty$
a.s. and, therefore, (cf. (\ref{f002}))  
$$
\lim_{z\to\infty}\psi_\lambda(z)=+\infty. 
$$ 
Consequently, letting $z\to +\infty$ in (\ref{n72}) we obtain
(\ref{n71}). Now (\ref{n701}) takes the form  

\begin{eqnarray*}
&&
\varphi_\lambda(0) 
=
-\frac 1 c \lp  \varphi_ \lambda(+\infty)-\varphi_ \lambda(0)\rp .
\end{eqnarray*}
This implies that $c=1$ since $\varphi_ \lambda(+\infty)=0$ by the assumption that $+\infty$ is 
natural  (cf. (\ref{f002})).
\hfill\break\hfill
{\sl b)} To proceed with the proof that $M^h$ is a martingale,   
consider first the case with continuously differentiable $h.$ Then,
applying (\ref{701}),
\begin{equation}
\label{71}
dM^h_t=h(L_t)(d\tilde Y_t+dL_t) +S(X_t)h'(L_t)dL_t-h(L_t)dL_t=h(L_t)dY_t,
\end{equation}
where we have used that
$$
S(X_t)h'(L_t)dL_t=S(0)h'(L_t)dL_t=0.
$$
Consequently, $M^h$ is a continuous local martingale, and it 
is a continuous martingale if for any $T>0$ the process 
$\{M^h_t\,:\, 0\leq t\leq T\}$ is
uniformly integrable (u.i.). To prove this, we use again (\ref{701}) and
write 
\begin{equation}
\label{72}
M^h_t=h(L_t)\tilde Y_t+h(L_t)L_t+1-H(L_t).
\end{equation}
Since $h$ is non-increasing and and has a compact support in $[0,K]$
we have 
$$
|h(L_t)L_t+1-H(L_t)|\leq K\,\sup_{x\in[0,K]} h(x)+\int_0^\infty h(u)du
$$
showing that  $\{ h(L_t)L_t+1-H(L_t): t\geq 0\}$ is u.i. Moreover, since  $\{
h(L_t): t\geq 0\}$ is bounded and $\{\tilde Y_t\,:\, 0\leq t\leq T\}$ is u.i. it
follows that $\{h(L_t)\tilde Y_t\,:\, 0\leq t\leq T\}$ is u.i.. Consequently,
$\{M^h_t\,:\, 0\leq t\leq T\}$ is u.i., as claimed, and, hence,
$\{M^h_t: t\geq 0\}$ is a true martingale implying (\ref{645}). 
By the monotone class theorem (see, e.g., Meyer \cite{meyer66} T20 p. 28) 
we can deduce that $\{M^h_t: t\geq 0\}$ remains a martingale if 
the assumtion 
``$h$ is continuously differentiable'' is relaxed to be ``$h$ is bounded and
Borel-measurable''. The proof of Theorem \ref{thm62} is now complete.
\end{proof}

\begin{example}
\label{ex61} Let $h(x):= {\bf 1}_{[0,\ell)}(x)$ with $\ell>0.$ Then
$$
h(0)=1,\quad \int_x^\infty h(y)dy=(\ell-x)^+,\quad \int_0^\infty
h(y)dy=\ell,
$$
and the martingale $M^h$ takes the form
\begin{eqnarray}
\label{68}
&&
\nonumber
M^h_u=\frac 1\ell\left(S(X_u){\bf 1}_{\{L_u<\ell\}}+ (\ell-L_u)_+\right)
\\
&&
\nonumber
\hskip.8cm
=\frac 1\ell\left(S(X_u)+ \ell-L_u\right){\bf 1}_{\{L_u<\ell\}}
\\
&&
\nonumber
\hskip.8cm
=\frac 1\ell\left(S(X_{u\wedge{\tau_\ell}})+ \ell-L_{u\wedge{\tau_\ell}}\right)
\\
&&
\nonumber
\hskip.8cm
=1+\frac 1\ell\left(S(X_{u\wedge{\tau_\ell}})-L_{u\wedge{\tau_\ell}}\right).
\\
&&
\nonumber
\hskip.8cm
=1+\frac 1\ell\,\tilde Y_{u\wedge{\tau_\ell}}.
\end{eqnarray}
\end{example}

\subsection{The law of $X$ 
under the penalized measure}
\label{sec63}

In this section we study the process $X$ under the penalized measure
$\Q$ introduced in Theorem \ref{thm62}. In fact, we consider a more
general situation, and assume that $h$ is ``only'' a Borel measurable and
non-negative function defined on $\R_+$ such that 
\begin{equation}
\label{685}
\int_0^\infty h(x)dx =1.
\end{equation}
For such a function $h$ we define
\begin{equation}
\label{69}
M^h_t:= S(X_t)h(L_t)+1-H(L_t),
\end{equation}
where
$$
H(x):=\int_0^x h(y)dy.
$$
It can be proved (see Roynette et al. \cite{rvyII} Section 3.2 and
\cite{rvyV} Section 3) that $\{M^h_t:t\geq 0\}$ is also in this more
general case a martingale such that $\E_0(M^h_t)=1$ and 
$\lim_{t\to\infty}M^h_t=0.$
Therefore, we may define, for each $u\geq 0,$ a probability measure  
$\Q^h$ on $(\cC,\cC_u)$ by setting 
\begin{equation}
\label{73}
\Q^h(\Lambda_u):=\E_0\left({\bf 1}_{\Lambda_u}\, M^h_u\right)\qquad \Lambda_u\in\cC_u.
\end{equation}
The notation $\E^h$ is used for the expectation with respect to
$\Q^h.$ Next two propositions constitute a generalization of Theorem
1.5 in \cite{rvyV}.

\begin{proposition}
\label{prop61}
Under $\Q^h,$ the random variable $L_\infty:=\lim_{t\to \infty}L_t$ is finite
a.s. and 
$$
\Q^h(L_\infty\in d\ell)=h(\ell)\,d\ell.
$$
\end{proposition}
\begin{proof}
For $u\geq 0$ and $\ell\geq 0$ it holds $\{L_u\geq \ell\}\in\cC_u,$
and, consequently, 
$$
\Q^h(L_u\geq \ell)=\E_0\left({\bf 1}_{\{L_u\geq \ell\}}\, M^h_u\right)
=\E_0\left({\bf 1}_{\{\tau_\ell\leq u\}}\, M^h_u\right).
$$
By optional stopping,
$$
\E_0\left({\bf 1}_{\{\tau_\ell\leq u\}}\, M^h_u\right)
=\E_0\left({\bf 1}_{\{\tau_\ell\leq u\}}\, M^h_{\tau_\ell}\right), 
$$
but
\begin{eqnarray}
&&
\label{735}
\nonumber
 M^h_{\tau_\ell}
=  S(X_{\tau_\ell})h(L_{\tau_\ell})+1-H(L_{\tau_\ell})
= S(0)h(\ell)+1-H(\ell)
\\
&&
\hskip.8cm
=\int_\ell^\infty h(y)\,dy.
\end{eqnarray}
As a result,
$$
\Q^h(L_u\geq \ell)=\left(\int_\ell^\infty h(y)\,dy\right)\,\P_0(\tau_\ell\leq u).
$$
Letting here $u\to \infty$ and using the fact that $\tau_\ell$ is
finite $\P_0-$a.s. shows that
$$
\Q^h(L_\infty\geq \ell)=\int_\ell^\infty h(y)\,dy.
$$
Moreover, from assumption (\ref{685}) it now follows that $L_\infty$ is
$\Q^h$-a.s. finite, and the proof is complete.
\end{proof}

In the proof of the next proposition we use the process $X^\uparrow=\{ X^\uparrow_t:t\geq 0\}$
which is obtained from $\widehat X$ (cf. (\ref{kill})) by conditioning $\widehat X$ not to
hit 0. The process  $X^\uparrow$ can be described as Doob's
$h$-transform of $\widehat X,$ see, e.g., Salminen, Vallois and Yor
\cite{SVY07} p.105. The probability measure and the expectation
operator associated with  $X^\uparrow$ are denoted by $\P^\uparrow$
and $\E^\uparrow,$ respectively. The transition density and the speed measure associated with $X^\uparrow$ 
are given by
\begin{equation}
\label{f01}
p^\uparrow(t;x,y):= \frac{\hat p(t;x,y)}{ S(y)S(x)},\quad\,m^\uparrow(dy):=S(y)^2\,m(dy).
\end{equation}
Notice (cf. (\ref{f00})) that
\begin{equation}
\label{f02}
p^\uparrow(t;0,y):= \lim_{x\downarrow 0}
p^\uparrow(t;x,y)
= 
\frac{f_{y0}(t)}{S(y)}.
\end{equation}
Consequently, we have the formula 
\begin{equation}
\label{f03}
1=\P_0^\uparrow\lp X^\uparrow_t> 0\rp =\int_0^\infty
p^\uparrow(t;0,y)\,m^\uparrow(dy)
=\int_0^\infty f_{y0}(t)\, S(y) \,m(dy).
\end{equation}


\begin{proposition}
\label{prop62}
Let $\lambda$ denote the last exit time from 0, i.e., 
$$
\lambda:=\sup\{t\,:\, X_t=0\}
$$
with $\lambda=0$ if $\{\cdot\}=\emptyset.$ Then 
\begin{description}
\item{{\bf 1)}} $\Q^h(0<\lambda<\infty)=1,$
\item{{\bf 2)}} under $\Q^h$
\begin{description}
\item{{\bf a)}} \hskip.2cm $\{X_t\,:\,t\leq \lambda\}$ and $\{X_{\lambda+t}\,:\,t\geq 0\}$
  are independent,
\item{{\bf b)}} \hskip.1cm conditionally on $L_\infty=\ell,$ the process 
$\{X_t\,:\,t\leq \lambda\}$ is distributed as $\{X_t\,:\,t\leq \tau_\ell\}$
  under $\P_0,$ in other words, 
\begin{eqnarray}
\label{74}
&&
\nonumber
\hskip-2cm
\E^h\left(F(X_t\,:\, t\leq \lambda)\,f(L_\infty)\right)
\\
&&
=
\int_0^\infty f(\ell)h(\ell)\E_0\left(F(X_t\,:\, t\leq
\tau_\ell\right)\, d\ell.
\end{eqnarray}
where $F$ is a bounded and measurable
  functional 
defined in the canonical space $(C,\cC,(\cC_t))$ and
$f: [0,\infty)\mapsto [0,\infty)$ is a bounded and measurable function. 
\item{{\bf c)}} \hskip.2cm the process 
$\{X_{\lambda+t}\,:\,t\geq 0\}$ is distributed as 
$\{X^\uparrow_t\,:\,t\geq 0 \}$ started from 0. 
\end{description}
\end{description}
\end{proposition}
\begin{proof} 
Consider for a given  $T>0$ 
\begin{eqnarray*}
&&
\Delta
 :=
\E^h\left(F_1(X_u\,:\, u\leq
\lambda)\,F_2(X_{\lambda+v}\,:\, v\leq T)\,f(L_\lambda)\,{\bf 1}_{\{0<\lambda<\infty\}} \right),
\end{eqnarray*}
where $F_1$ and $F_2$ are bounded and measurable
  functionals 
defined in the canonical space $(C,\cC,(\cC_t))$ 
and
$f: [0,\infty)\mapsto [0,\infty)$ is a bounded and measurable function. 
For $N>0$ define
$$
\lambda_N:=\sup\{u\leq N\,:\,X_u=0\}
$$
and
$$
\Delta^{(1)}_N
:=\E^h\left(F_1(X_u\,:\, u\leq
\lambda_N)\,F_2(X_{\lambda_N+v}\,:\, v\leq T)\,f(L_{\lambda_N})
\,{\bf 1}_{\{\lambda_N+T<N\}} \right). 
$$
Then 
$$
\Delta=\lim_{N\to\infty}\Delta^{(1)}_N.
$$
By absolute continuity, cf. (\ref{73}),
\begin{eqnarray*}
&&
\hskip-.5cm
\Delta^{(1)}_N
=\E_0\left(F_1(X_u\,:\, u\leq
\lambda_N)\,F_2(X_{\lambda_N+v}\,:\, v\leq T)\,f(L_{\lambda_N})\,{\bf
  1}_{\{\lambda_N+T<N\}}\, M^h_N\rp
\\
&&
\hskip.4cm
=
\E_0\Big(F_1(X_u\,:\, u\leq
\lambda_N)\,F_2(X_{\lambda_N+v}\,:\, v\leq T)\,f(L_{\lambda_N})\,{\bf 1}_{\{\lambda_N+T<N\}}
\\
&&
\hskip5cm
\times
\lp S(X_N)h(L_N)+1-H(L_N)\rp\Big).
\end{eqnarray*}
Since $F_1, F_2,$ and $f$ are bounded and 
$$
\lim_{N\to\infty}\lp 1-H(L_N)\rp =0\quad \P_0{\text -a.s.}
$$
we have 
\begin{eqnarray*}
&&
\hskip-.5cm
\Delta=
\lim_{N\to\infty}
\E_0\Big(F_1(X_u\,:\, u\leq
\lambda_N)\,F_2(X_{\lambda_N+v}\,:\, v\leq T)\,f(L_{\lambda_N})
\\
&&
\hskip4cm
\times{\bf  1}_{\{\lambda_N+T<N\}}
 \,S(X_N)h(L_N)\Big).
\end{eqnarray*}
Let $\Delta^{(2)}_N$ denote the expression after the limit sign. Then
we write
\begin{eqnarray*}
&&
\hskip-.5cm
\Delta^{(2)}_N=
\E_0\Big( \sum_{\ell} F_1(X_u\,:\, u\leq
\tau_{\ell-})
\,F_2(X_{\tau_{\ell-}+v}\,:\, v\leq T)
\\
&&
\hskip4cm
\times f(\ell)\,
\,{\bf  1}_{\{\tau_{\ell-}+T< N<\tau_{\ell}\}}
\,S(X_N)h(\ell)\Big),
\end{eqnarray*}
where $\{\tau_\ell\}$
is the right continuous inverse of $\{L_t\}$  (see (\ref{e00})). 
By the Master formula (see Revuz and Yor \cite{revuzyor01} p. 475 and 483)
\begin{eqnarray*}
&&
\hskip-1.5cm
\Delta^{(2)}_N=
\int_0^\infty d\ell\, h(\ell)f(\ell)\E_0\Big( F_1(X_u\,:\, u\leq
\tau_{\ell})
\\
&&
\hskip.4cm
\times\,\int_{\cE}{\bf n}(de)\,
\,F_2(e_v\,:\, v\leq T)\,{\bf  1}_{\{T\leq N-\tau_{\ell}\leq \zeta(e)\}}\,S(e_{N-\tau_\ell})\Big),
\end{eqnarray*}
where $\cE$ denotes the excursion space, $e$ is a generic excursion,
$\zeta(e)$ is the life time of the excursion $e,$ and ${\bf n}$ is the
It\^o measure in the excursion space (see, e.g.,  \cite{revuzyor01}
p. 480 and \cite{SVY07}). We claim that 
\begin{eqnarray}
\label{745}
&&
\nonumber
\hskip-2cm
I:=\int_{\cE}
\,F_2(e_v\,:\, v\leq T)\,{\bf  1}_{\{T\leq T'\leq \zeta(e)\}}\,S(e_{T'})\,{\bf n}(de)
\\
&&
\hskip4cm
=\E_0^\uparrow\lp F_2(X_v\,:\, v\leq T)\rp.
\end{eqnarray}
Notice that the right hand side of (\ref{745}) does not depend on
$T'.$ We prove (\ref{745}) for $F_2$ of the form 
$$
F_2(e_v\,:\, v\leq T)=G(e_{t_1},\dots, e_{t_k}), \quad
t_1<t_2<\dots<t_k=T,
$$
where $G$ is a bounded and measurable function. For simplicity, take
$k=2$ and use Theorem 2 in \cite{SVY07} to obtain (for notation and
results needed, see
(\ref{kill2}), (\ref{f00}), (\ref{f01}) and (\ref{f02}))
\begin{eqnarray*}
&&
\hskip-1cm
I=\int_{[0,\infty)^3}f_{x_1,0}(t_1)\,\hat p(t_2-t_1;x_1,x_2)\,\hat
  p(T'-t_2;x_2,x_3)
\\
&&
\hskip3cm
\times\,
G(x_1,x_2)\, S(x_3)\,m(dx_1)\,m(dx_2)\,m(dx_3)
\\
&&
\hskip-.5cm
=\int_{[0,\infty)^3}S(x_1)\,f_{x_1,0}(t_1)\, p^\uparrow(t_2-t_1;x_1,x_2)\,\hat
  p^\uparrow(T'-t_2;x_2,x_3)
\\
&&
\hskip2cm
\times
\,G(x_1,x_2)\, S(x_2)^2\,S(x_3)^2\,m(dx_1)\,m(dx_2)\,m(dx_3)
\\
&&
\hskip-.5cm
=\E_0^\uparrow\lp G(X_{t_1},X_{t_2})\rp
\end{eqnarray*}
proving (\ref{745}). Consequently, we have (for all $N$)
\begin{eqnarray*}
&&
\hskip-1.5cm
\Delta^{(2)}_N=\E_0^\uparrow\lp F_2(X_v\,:\, v\leq T)\rp
\int_0^\infty d\ell\, h(\ell)f(\ell)\E_0\lp F_1(X_u\,:\, u\leq
\tau_{\ell})\rp,
\end{eqnarray*}
and choosing here $F_1,$ $F_2,$ and $f$ appropriately implies
all the claims of Proposition. In particular, $F_1=F_2=1$ and
$f=1$  yields $\Q^h_0(0<\lambda<\infty)=1,$ and, hence,
$L_\infty=L_\lambda$ $\Q^h_0$-a.s.  
\end{proof}

\section{Appendix: a technical lemma}
\label{20}

The following lemma could be viewed as a ``weak'' form of the Tauberian
theorem (cf. Feller \cite{feller71} Theorem 1 p. 443) stating, roughly
speaking, that if
two functions  behave similarly at zero then their Laplace transforms  
behave similarly at infinity.

\begin{lemma}
\label{lemma31}
Let $\mu$ be a $\sigma$-finite measure on $[0,+\infty)$ and $g_1$ and $g_2$ two real valued functions such that
for some $\la_0>0$
$$
C_i:=\int_{[0,+\infty)}{\rm e}^{-\la_0 \gamma}\,|g_i(\gamma)|\,\mu(d\gamma)<\infty,\quad i=1,2.
$$
Assume also that $g_2(\gamma)>0$ for all $\gamma.$ Introduce for $\lambda\geq \la_0$
$$
f_i(\lambda):=\int_{[0,+\infty)}{\rm e}^{-\la \gamma}\,g_i(\gamma)\,\mu(d\gamma),\quad i=1,2.
$$
and suppose
\begin{equation}
\label{c1}
\lim_{\la\to +\infty}f_2(\la)\,{\rm e}^{b\la}=+\infty\qquad {\rm for\ all\ } b>0. 
\end{equation}
Then 
\begin{equation}
\label{c2}
 {\disp{ g_1(\gamma)\,\mathop{\sim}_{\gamma\to 0}\, g_2(\gamma)}}
\end{equation}
implies
\begin{equation}
\label{c3}
f_1(\la)
\,\mathop{\sim}_{\la\to +\infty}\, 
f_2(\la)
\end{equation}
\end{lemma}
\begin{proof} By property (\ref{c2})
there exist two functions $\te_*$ and $\te^*$ 
such that for some $\ep>0$  and for all $\gamma\in(0,\ep)$ 
\begin{equation}
\label{c4}
\te_*(\ep)\,g_2(\gamma)\leq g_1(\gamma)\leq \te^*(\ep)\,g_2(\gamma).
\end{equation}
and 
\begin{equation}
\label{c41}
\lim_{\ep\to 0}\te_*(\ep)=\lim_{\ep\to 0}\te^*(\ep)=1.
\end{equation}
We assume also that  $\te_*(\ep)>0$ and $\te^*(\ep)>0.$ 
Letting $\la\geq 2\la_0$ we have for
$\ga\geq \ep$
$$
\la\ga\geq \la_0\ga+\frac{\la\ga} 2\geq \la_0\ga+\frac{\la\ep} 2
$$
and  
\begin{equation}
\label{c5}
\int_\ep^\infty {\rm e}^{-\la \gamma}\,|g_i(\gamma)|\,\mu(d\gamma)
\leq {\rm e}^{-\la\ep/2} \,\int_\ep^\infty\, {\rm e}^{-\la_0
  \gamma}\,|g_i(\gamma)|\,\mu(d\gamma)\leq 
 {\rm e}^{-\la\ep/2}\,C_i.
\end{equation}
Furthermore, from (\ref{c4}) 
\begin{eqnarray}
\label{c6}
&&
\nonumber
\hskip-1.7cm
\int_0^\ep {\rm e}^{-\la \gamma}\,g_1(\gamma)\,\mu(d\gamma)
\leq
\te^*(\ep)\,\int_0^\ep {\rm e}^{-\la \gamma}\,g_2(\gamma)\,\mu(d\gamma)
\\
&&
\nonumber
\hskip2cm
\leq
\te^*(\ep)\,\int_0^\infty {\rm e}^{-\la \gamma}\,g_2(\gamma)\,\mu(d\gamma)
\\
&&
\hskip2cm
\leq 
\te^*(\ep)\,f_2(\la)
\end{eqnarray}
since $g_2$ is assumed to be positive. Writing
$$
f_1(\la)=\int_0^\ep {\rm e}^{-\la \gamma}\,g_1(\gamma)\,\mu(d\gamma)+\int_\ep^\infty {\rm e}^{-\la \gamma}\,g_1(\gamma)\,\mu(d\gamma)
$$
the estimates in (\ref{c5}) and (\ref{c6}) yield 
$$
f_1(\la)\leq \te^*(\ep)\,f_2(\la)+ {\rm e}^{-\la\ep/2}\,C_1,
$$
which after dividing with  $f_2(\la)>0$ implies using (\ref{c1}) and (\ref{c41})
\begin{equation}
\label{c7}
\limsup_{\la\to+\infty}\frac{f_1(\la)}{f_2(\la)}=1.
\end{equation}
For a lower bound, consider
\begin{eqnarray*}
&&
\hskip-1.7cm
f_1(\la)=\int_{[0,\infty)} {\rm e}^{-\la \gamma}\,g_1(\gamma)\,\mu(d\gamma)
\\
&&
\hskip-.6cm
\geq 
\int_{[0,\ep)} {\rm e}^{-\la \gamma}\,g_1(\gamma)\,\mu(d\gamma)
-\int_\ep^\infty {\rm e}^{-\la \gamma}\,|g_1(\gamma)|\,\mu(d\gamma)
\\
&&
\hskip-.6cm
\geq 
\te_*(\ep)\,\int_{[0,\ep)} {\rm e}^{-\la \gamma}\,g_2(\gamma)\,\mu(d\gamma)
- {\rm e}^{-\la\ep/2}\,C_1
\\
&&
\hskip-.6cm
\geq 
\te_*(\ep)\left( f_2(\te)-\int_\ep^\infty {\rm e}^{-\la \gamma}\,g_2(\gamma)\,\mu(d\gamma)\right)
- {\rm e}^{-\la\ep/2}\,C_1.
\\
&&
\hskip-.6cm
\geq 
\te_*(\ep)\,f_2(\te)- \te_*(\ep)\,{\rm e}^{-\la\ep/2}\,C_2- {\rm e}^{-\la\ep/2}\,C_1.
\end{eqnarray*}
Hence,
$$
\frac{f_1(\la)}{f_2(\la)}\geq \te_*(\ep)-\left(\te_*(\ep)C_2-C_1\right)
\frac 1{{\rm  e}^{\la\ep/2}\,f_2(\la)}
$$
showing that 
$$
\liminf_{\la\to+\infty}\frac{f_1(\la)}{f_2(\la)}\geq 1,
$$
and completing the proof.
\end{proof}


\bibliographystyle{plain}
\bibliography{yor1}
\end{document}